\documentclass[a4paper,11pt]{article}

\usepackage[T1]{fontenc}
\usepackage[latin9]{inputenc}
\usepackage{geometry}
\usepackage{amsmath}
\usepackage{amsthm}
\usepackage{amssymb}
\usepackage{bbm}
\usepackage{color}
\usepackage{paralist}
\usepackage{todonotes}
\usepackage{mathtools}
\usepackage[unicode=true,pdfusetitle,
 bookmarks=true,bookmarksnumbered=false,bookmarksopen=false,
 breaklinks=false,pdfborder={0 0 0},pdfborderstyle={},backref=false,colorlinks=false]
 {hyperref}

\makeatletter

\usepackage[nameinlink,capitalise,noabbrev]{cleveref}

\hypersetup{%
    bookmarksnumbered, bookmarksopen=true, bookmarksopenlevel=1,%
}

\theoremstyle{plain}
\newtheorem{thm}{Theorem}[section]
\crefname{thm}{Theorem}{Theorems}
\theoremstyle{plain}
\newtheorem{lem}[thm]{Lemma}
\crefname{lem}{Lemma}{Lemmas}
\theoremstyle{plain}

\theoremstyle{plain}
\newtheorem*{claim*}{Claim}
\crefname{claim}{Claim}{Claims}
\theoremstyle{definition}
\newtheorem{defn}[thm]{Definition}
\theoremstyle{plain}
\newtheorem{conjecture}[thm]{Conjecture}
\crefname{conjecture}{Conjecture}{Conjectures}
\theoremstyle{plain}
\newtheorem{prop}[thm]{Proposition}
\crefname{prop}{Proposition}{Propositions}
\theoremstyle{definition}

\theoremstyle{definition}
\newtheorem*{rem*}{Remark}
\theoremstyle{plain}

\theoremstyle{plain}

\theoremstyle{observation}

\crefformat{equation}{#2(#1)#3}

\date{}

\crefformat{equation}{#2(#1)#3}

\let\originalleft\left
\let\originalright\right
\renewcommand{\left}{\mathopen{}\mathclose\bgroup\originalleft}
\renewcommand{\right}{\aftergroup\egroup\originalright}
\usepackage{verbatim}

\makeatletter
\renewcommand*{\UrlTildeSpecial}{%
  \do\~{%
    \mbox{%
      \fontfamily{ptm}\selectfont
      \textasciitilde
    }%
  }%
}%
\let\Url@force@Tilde\UrlTildeSpecial
\makeatother

\let\OLDthebibliography\thebibliography
\renewcommand\thebibliography[1]{
  \OLDthebibliography{#1}
  \setlength{\parskip}{0pt}
  \setlength{\itemsep}{3pt plus 0.3ex}
}

\makeatother

\allowdisplaybreaks
\numberwithin{equation}{section}

\usepackage[labelfont=bf,labelsep=period]{caption}

\begin{document}
\title{Universality of random permutations}
\author{Xiaoyu He \thanks{Department of Mathematics, Stanford University, Stanford, CA 94305.
Email: \href{alkjash@stanford.edu}{\nolinkurl{alkjash@stanford.edu}}.
Research supported by a NSF Graduate Research Fellowship.} 
\and Matthew Kwan \thanks{Department of Mathematics, Stanford University, Stanford, CA 94305.
Email: \href{mattkwan@stanford.edu}{\nolinkurl{mattkwan@stanford.edu}}.
Research supported in part by SNSF project 178493.}}

\maketitle
\global\long\def\RR{\mathbb{R}}%
\global\long\def\FF{\mathbb{F}}%
\global\long\def\QQ{\mathbb{Q}}%
\global\long\def\E{\mathbb{E}}%
\global\long\def\Var{\operatorname{Var}}%
\global\long\def\CC{\mathbb{C}}%
\global\long\def\NN{\mathbb{N}}%
\global\long\def\ZZ{\mathbb{Z}}%
\global\long\def\GG{\mathbb{G}}%
\global\long\def\tallphantom{\vphantom{\sum}}%
\global\long\def\tallerphantom{\vphantom{\int}}%
\global\long\def\supp{\operatorname{supp}}%
\global\long\def\one{\mathbbm{1}}%
\global\long\def\d{\operatorname{d}}%
\global\long\def\Unif{\operatorname{Unif}}%
\global\long\def\Po{\operatorname{Po}}%
\global\long\def\Bin{\operatorname{Bin}}%
\global\long\def\Ber{\operatorname{Ber}}%
\global\long\def\Geom{\operatorname{Geom}}%
\global\long\def\Rad{\operatorname{Rad}}%
\global\long\def\floor#1{\left\lfloor #1\right\rfloor }%
\global\long\def\ceil#1{\left\lceil #1\right\rceil }%
\global\long\def\falling#1#2{\left(#1\right)_{#2}}%
\global\long\def\cond{\,\middle|\,}%
\global\long\def\su{\subseteq}%
\global\long\def\row{\operatorname{row}}%
\global\long\def\col{\operatorname{col}}%
\global\long\def\spn{\operatorname{span}}%
\global\long\def\eps{\varepsilon}%
\global\long\def\S{\mathcal{S}}%

\global\long\def\mk#1{\textcolor{blue}{\textbf{[MK comments:} #1\textbf{]}}}
\global\long\def\af#1{\textcolor{red}{\textbf{[AF comments:} #1\textbf{]}}}
\global\long\def\xh#1{\textcolor{cyan}{\textbf{[XH comments:} #1\textbf{]}}}

\begin{abstract}
It is a classical fact that for any $\varepsilon>0$,
a random permutation of length $n=\left(1+\varepsilon\right)k^{2}/4$
typically contains an increasing subsequence of length $k$. As a far-reaching
generalization, Alon conjectured that a random permutation of this same length $n$ is typically \emph{$k$-universal}, meaning that it simultaneously contains
every pattern of length $k$. He also made the simple observation
that for $n=O(k^{2}\log k)$, a random length-$n$ permutation is typically $k$-universal. We make the first significant progress towards Alon's conjecture by showing that $n=2000k^{2}\log\log k$ suffices.
\end{abstract}

\section{Introduction}

A mathematical structure is said to be \emph{universal} if it contains all possible substructures, in some specified sense. This notion may have been first considered in a 1964 paper by Rado~\cite{Rad64}, in which he found examples of graphs, simplicial complexes and functions which are universal in various ways. Another famous universal structure is a \emph{de Bruijn sequence} (with parameters $k$ and $q$, say), which is a string over a size-$q$ alphabet in which every possible length-$k$ string appears exactly once as a substring.

One topic that has received particular attention over the years is the case of universality for finite graphs. We say that a graph is $k$-\emph{universal} (or $k$-\emph{induced-universal}) if it contains every graph on $k$ vertices as an induced subgraph. The problems that have received the most interest in this area are (1) to find a $k$-universal graph with as few vertices as possible, and (2) to understand for which $n$ a ``typical'' $n$-vertex graph is $k$-universal. These problems are related to the problem of finding optimal \emph{adjacency labeling schemes} in theoretical computer science; for more details we refer the reader to \cite{alstrup2015adjacency} and the references therein.

In an exciting recent paper by Alon~\cite{Alo17}, both of these problems were effectively resolved. He showed with a probabilistic proof that there exists a $k$-universal graph with $(1+o(1))2^{(k-1)/2}$ vertices, asymptotically matching a lower bound due to Moon~\cite{moon1965minimal}. Alon also showed that as soon as $n$ is large enough that a random $n$-vertex graph typically contains a $k$-vertex clique and a $k$-vertex independent set, then such a random graph is typically also $k$-universal. His proofs involved a classification of graphs according to their numbers of automorphisms, taking advantage of the fact that graphs with few automorphisms are easier to embed into random graphs.

Alon's work essentially closes the book on the study of $k$-universal graphs, but substantial challenges remain in many other settings. One important example is the case of permutations, where there is no natural notion of an automorphism, and no natural scheme to embed sub-permutations using ``quasirandomness'' conditions. Let $\S_n$ be the set of all permutations of the $n$-element set $[n]:=\{1,\dots,n\}$. We say that a permutation $\sigma\in\S_{n}$ \emph{contains}
a \emph{pattern} $\pi\in\S_{k}$, and write $\pi \in \sigma$, if there are indices $1\le x_1<\dots<x_k\le n$
such that for $1\le i,j\le k$ we have $\sigma\left(x_{i}\right)<\sigma\left(x_{j}\right)$
if and only if $\pi\left(i\right)<\pi\left(j\right)$. Say that
$\sigma$ is \emph{$k$-universal} or a \emph{$k$-superpattern} if
it contains every $\pi\in\S_{k}$. As before, there are two main directions to consider: (1) finding the shortest possible $k$-universal permutation and (2) understanding for which $n$ a typical length $n$ permutation is $k$-universal.

As a simple lower bound for both problems, note that if $\sigma\in \S_n$ is $k$-universal, then we must have $\binom n k\ge k!$, since $\sigma$ contains $k!$ distinct patterns.  Using Stirling's approximation and the fact $\binom{n}{k}\leq n^k/k!$, we deduce the lower bound $$n\geq \left(\frac1{e^2}-o(1)\right)k^2.$$ For the first problem (of finding short $k$-universal permutations), this lower bound is not too far from best-possible: Miller~\cite{Mil09} constructed a $k$-universal permutation with length  $n\leq (1/2+o(1))k^2$, and the $o\left(1\right)$-term
was recently improved by Engen and Vatter~\cite{EV18}. This constant $1/2$ was conjectured to be tight by Eriksson, Eriksson, Linusson
and W\"astlund~\cite{EELW07}, while the constant $1/e^2$ from the lower bound was conjectured to be tight by Arratia~\cite{Arr99}.

Regarding universality of \emph{random} permutations, much less is known. Note that containing the identity permutation $1_k\in \S_k$ is equivalent to containing an increasing sequence of length $k$, and the longest increasing subsequence of a typical $\sigma\in\S_n$ is known\footnote{This fact is actually surprisingly difficult to prove, and is due independently to Logan and Shepp~\cite{LS77} and to Vershik and Kerov~\cite{VK77}
(see also \cite{AD95}). The study of increasing subsequences in random permutations has a rich history, see for example the survey \cite{Rom15}.} to be of length $(2+o(1))\sqrt n$. It follows that we cannot hope for a typical $\sigma\in \S_n$ to be $k$-universal unless $n\ge (1/4+o(1))k^2$. In 1999, Alon made the following striking conjecture (see \cite{Alo16,Arr99}).
\begin{conjecture}
\label{conj:noga}For a fixed $\varepsilon>0$, a random permutation of length $(1+\varepsilon)k^2/4$ is w.h.p.\footnote{We say that an event holds \emph{with high probability}, or w.h.p. for short,
if it holds with probability $1-o\left(1\right)$. Here and for the
rest of the paper, asymptotics are as $k\to\infty$ and/or $n\to\infty$.}\ $k$-universal.
\end{conjecture}

Intuitively, \cref{conj:noga} can be justified by comparison to universality in graphs: in much the same way that cliques and independent sets are the ``hardest'' subgraphs to find in a random graph, it is believed that monotonically increasing and decreasing patterns are the hardest patterns to find in a random permutation. We also remark that \cref{conj:noga} contradicts the aforementioned conjecture by Eriksson, Eriksson, Linusson
and W\"astlund.

In the ``ordered'' setting of random permutations, most of the standard tools used in the unordered setting of graphs are not applicable, and \cref{conj:noga} seems rather challenging to prove. Indeed, in a recent discussion of the problem, Alon~\cite{Alo16} highlighted the more modest problem of simply showing that for $n=1000k^2$ a typical $\sigma\in \S_k$ is $k$-universal. He also observed a simple upper bound of the form $n=O(k^2\log k)$ (we will sketch a proof of this in \cref{sec:outline}). Our main result is the following substantial improvement.

\begin{thm}
\label{thm:main}A random permutation of length $2000k^2\log\log k$ is w.h.p.\ $k$-universal.
\end{thm}

Since there is no natural notion of symmetry for permutations, we were not able to take quite the same approach as Alon took for the graph case. However, the proof of \cref{thm:main} still proceeds via a ``structure-vs-randomness''
dichotomy (see \cite{Tao07b} for a discussion of this phenomenon in general).
In our proof of \cref{thm:main} we show that every $\pi\in\mathcal{S}_{k}$
can be decomposed into a ``structured part'' and a ``quasirandom
part''. The ``structured part'' of $\pi$ is likely to appear in $\sigma$ for one reason, and the ``quasirandom part'' is likely to appear for a different reason. We outline the proof in more detail in \cref{sec:outline}, but it is worth mentioning here that because \emph{most} permutations are entirely quasirandom in our sense, the following theorem also follows from our proof approach.

\begin{thm}
\label{thm:almost-all}For any $k\ge 1$, there is a set $\mathcal{Q}_{k}\subseteq\mathcal{S}_{k}$
of $\left(1-o\left(1\right)\right)k!$ length-$k$ permutations such that w.h.p.\ a random permutation of length $20k^2$ contains every $\pi \in \mathcal{Q}_k$.
\end{thm}

The definition of the set $\mathcal{Q}_k$ is too technical to describe here, but it is completely explicit, see \cref{sec:quasirandom}. We made no attempt to optimize the constants in \cref{thm:main,thm:almost-all},
but we believe new ideas would be required to push $n$ down to $\left(1+o\left(1\right)\right)k^{2}/4$
in \cref{thm:almost-all}.

Of course, \cref{conj:noga,thm:main,thm:almost-all} can also be interpreted in terms of \emph{counting} $k$-universal permutations, and one natural avenue towards \cref{conj:noga} is to study the number of permutations $\sigma\in \S_n$ avoiding a specific pattern $\pi \in \S_k$. To be precise, given $\pi\in \S_k$, we let 
$$\S_n(\pi):=\{\sigma \in \S_n \mid  \sigma \textrm{ does not contain } \pi\}.$$ 
If we could prove that $|\S_n(\pi)|= o\left(n!/k!\right)$ for all $\pi \in \S_k$, it would follow that there are at least $n!-o(n!)$ permutations in $\S_n$ which are $k$-universal, and therefore that w.h.p. a random permutation of length $n$ is $k$-universal. The problem of estimating $|\S_n(\pi)|$ has a long and rich history, largely in the regime where $k$ is fixed and $n$ is large (we refer the reader to the survey \cite{claesson2008classification}, the book \cite{bona2016combinatorics}, and the references therein). The most important result in this area is due to Marcus and Tardos \cite{marcus2004excluded}, who resolved a conjecture of Stanley and Wilf, showing that for every $\pi \in \S_k$ there exists $c_{\pi}$ for which $|\S_n(\pi)|\leq c_{\pi}^n$. Note that for a given $k$, if we let $c_k:=\max\{c_{\pi}\mid \pi\in \S_k\}$, then by the result of Marcus and Tardos we obtain that there are at most $k!c_k^n$ permutations from $\S_n$ that avoid \emph{some} pattern of length $k$.

One may naively hope to prove new bounds for \cref{conj:noga} via bounds on $c_k$, but unfortunately this is hopeless. Fox~\cite{Fox13} showed that the dependence of $c_k$ on $k$ is extremely poor: $c_k=2^{\Omega(k^{1/4})}$. Nevertheless, it is still plausible that one may be able to prove $|\S_n(\pi)|=o(n!/k!)$ in the special case where $k$ is about $\sqrt n$, and this idea guides our proofs of \cref{thm:main,thm:almost-all}. In fact, it is possible to strengthen \cref{thm:almost-all} to prove the strong bound $|S_n(\pi)| \le n! e^{-\Omega(k^{5/4})}$ for $\pi \in \mathcal{Q}_k$, see \cref{sec:concluding} for details.

\subsection{Discussion and proof outline}\label{sec:outline}

Before describing our approach, we make a very convenient technical observation. For any $q\in \NN$, taking $m=\floor{n/(2q)}$, it is possible to couple a uniform random $\sigma \in \mathcal{S}_n$ with a uniform random $q\times m$ zero-one matrix $M$ (whose entries are independently zero or one with probability $1/2$), in such a way that $\sigma$ contains $\pi$ whenever $M$ ``contains'' $\pi$. Here we say a matrix $M$ contains $\pi$ if one can delete columns and rows, and change ones to zeros, to obtain the permutation matrix $P_\pi$ of $\pi$. Say that $M$ is $k$-universal if it contains all $k$-permutations, so that $\sigma$ is $k$-universal if $M$ is $k$-universal. One should think of $M$ as a reduced version of the permutation matrix $P_\sigma$ of $\sigma$, modified so that the entries of $M$ are independent. For the details of the coupling see \cref{sec:multi-threaded-scanning}.

To illustrate the utility of $M$, we start by sketching a proof of the (previously known) fact that for some constant $C$ and $n=Ck^2\log k$, a typical $\sigma\in \S_n$ is $k$-universal. Let $q=k$, so that $m=\lfloor n/(2k)\rfloor$ and $M$ is a uniform random $k\times m$ zero-one matrix. The idea is to consider a simple greedy algorithm that scans through $M$ attempting to find a copy of $\pi\in \S_k$. We will see that this algorithm fails with probability $e^{-\Omega(m)}=e^{-\Omega(n/k)}$, meaning that $|\mathcal{S}_n(\pi)| \le n!e^{-\Omega(n/k)}$. So if $n=Ck^2\log k$ for large $C$ then we can just sum over all $k!=e^{\Theta(k\log k)}$ possibilities for $\pi$.

Here is how the algorithm works. For $M$ to contain a permutation matrix $P_\pi$ means that there are indices $i_1<\dots<i_k$ such that $M(\pi(j),i_j)=1$ for each $j$. The algorithm proceeds in the simplest possible way: we scan through the indices $i=1,\dots,m$ one-by-one, repeatedly querying whether $M(\pi(1),i)=1$. Once we succeed in finding $i_1$ with $M(\pi(1),i_1)=1$ we continue running through the indices $i=i_1+1,\dots,m$, now querying whether $M(\pi(2),i)=1$ until we find $i_2$, and so on. If at least $k$ of the queries succeed during this algorithm, then it successfully finds a copy of $\pi$. Since each of the queries is independent and has success probability $1/2$, a straightforward Chernoff bound shows that the algorithm succeeds to find $\pi$ with probability $1-e^{-\Omega(m)}$, as desired.

A crucial observation about this algorithm is that regardless of whether or not it finds a copy of $\pi$, it exposes only a very small amount of information about $M$. We can imagine the algorithm tracing a ``thread'' through $M$, exposing at most one entry per column, and leaving the other entries completely untouched. The hope is to run our greedy algorithm several times to look for $\pi$ in slightly different ways, tracing different threads through $M$.

We do this as follows. Instead of taking $q=k$ we can take $q=2k$ (so $M$ has $2k$ rows), and then use our greedy algorithm to attempt to find $\pi$ in rows $1+t,\dots,k+t$, for several different choices of $t\in [k]$. That is, we scan through multiple ``shifted'' threads in the same matrix, and if any of our threads succeeds, we have that $\pi \in \sigma$.

The aim is to judiciously choose the thread indices $t$ in such a way that the threads are mostly disjoint, meaning that the searches are mostly independent of each other. If this were possible, it would allow us to amplify the probability that a single thread fails, thereby giving much stronger bounds on $|\mathcal{S}_n(\pi)|$ and thus proving $k$-universality for a smaller value of $n$.

This plan fails for two different reasons. The first is that, since we are concerned with very small probabilities of order $e^{-\Omega(n/k)}$, we cannot rule out the event that there is a very long run of zeros in some row of $M$. Indeed, the probability that a single row is entirely zero is also of order $e^{-\Omega(n/k)}$. Such a run of zeros would be simultaneously disastrous for multiple threads at once, since many threads could heavily intersect in that row. The second issue is that if a permutation is very ``self-similar'' then two different threads can ``synchronize''. For example, if there are two long sequences of indices $a_1<\dots<a_L$ and $b_1<\dots<b_L$ such that $\pi\left(a_{i}\right)=\pi\left(b_{i}\right)+\Delta$
for each $i$ (we call this situation a \emph{$\Delta$-shift} of length $L$ in $\pi$), then for any $t$, the part of thread $t$ that searches for $\pi(b_1),\dots,\pi(b_L)$ could coincide with the part of thread $t+\Delta$ that searches for $\pi(a_1),\dots,\pi(a_L)$. 

It is actually quite simple to overcome the first of these two issues because the appearance of long runs in $M$ is unlikely in absolute terms (in a typical outcome of $M$, the longest horizontal run of zeros has length  $O(\log k)$). We can simply define an event $\mathcal A$ that there are no long runs of zeros, show that $\Pr(\mathcal A)=1-o(1)$, and analyze our multi-threaded scanning procedure in the \emph{conditional} probability space where $\mathcal A$ holds, taking a union bound over all $\pi$ only in this conditional space. Note that this conditioning means that our approach no longer directly gives bounds on the number of $\pi$-free permutations $|\mathcal{S}_n(\pi)|$.

The second of the aforementioned issues is more serious, but it is only a problem if $\pi$ contains a long $\Delta$-shift, and it turns out that long $\Delta$-shifts are quite atypical. Indeed, it is possible to define a set $\mathcal Q_k$ of $(1-o(1))k!$ ``quasirandom'' permutations $\pi$ which have no long $\Delta$-shifts, so that no two threads can ``synchronize'' too much. This yields a proof of \cref{thm:almost-all}, the details of which are in \cref{sec:quasirandom}.

Now, if $\pi$ is non-quasirandom to such an extent that multi-threaded scanning is completely ineffective, then it must have long $\Delta$-shifts for many $\Delta$, which heavily constrains the structure of $\pi$. We might hope that there are very few non-quasirandom permutations, so that the basic bound $|\mathcal{S}_n(\pi)| \le n!e^{-\Omega(n/k)}$ suffices for a union bound over all non-quasirandom $\pi$, for some $n$ much smaller than $k^2\log k$. While this approach can yield a small constant-factor improvement, the number of non-quasirandom permutations is unfortunately still too large: for example, there are $(k/2)!=e^{\Theta(k\log k)}$ permutations $\pi$ satisfying $\pi(i)=i$ for $i\le k/2$.

Instead, our approach is as follows. We define a notion of a ``structured map'' $\phi:Z\to [k]$, where $Z\subseteq [k]$, in such a way that there are only $e^{O(k\log \log k)}$ different structured maps (in contrast to the $e^{\Theta(k\log k)}$ many permutations $\pi\in \mathcal S_k$). We then prove that every permutation $\pi\in \mathcal S_k$ can be partitioned into a quasirandom part and a structured part, in the sense that there is a partition $[k]=Q\cup Z$ such that the restriction $\pi|_Z:Z\to [k]$ is a structured map, and the restriction $\pi|_Q:Q\to [k]$ is in some sense quasirandom with respect to $\pi$.

Of course, since there are few structured maps, it would be straightforward to use the union bound to prove that a random permutation $\sigma\in \mathcal S_n$ typically contains every structured map, for some $n=O(k^2\log \log k)$. But since we need to handle ``hybrid'' permutations that may have their quasirandom and structured parts arbitrarily interleaved, this approach is insufficient. Instead we show that $M$ typically has a technical property we call $\mathcal{B}$ that for any set of positions in $M$ corresponding to a copy of $\phi$, the average length of the runs of zeros starting at these positions is $O(\log \log k)$, which is much shorter than the bound $O(\log k)$ guaranteed by $\mathcal A$ for individual runs. Once we condition on $\mathcal A\cap \mathcal B$, we then encounter no problems analyzing the multi-threaded scanning algorithm, yielding a proof of \cref{thm:main}. The details are in \cref{sec:main}.

\section{\label{sec:multi-threaded-scanning}Multi-threaded scanning}

Our first lemma reduces permutation universality to a notion of matrix universality. It will be useful to define the notion of an \emph{interval minor} introduced by Fox~\cite{Fox13}, which generalizes permutation containment.

\begin{defn}
The {\it interval contraction} of a pair of consecutive rows (resp. columns) in a zero-one matrix replaces those rows (resp. columns) by their entrywise binary OR. If $P$ and $M$ are two zero-one matrices, then $P$ is an {\it interval minor} of $M$ if it can be obtained from $M$ by repeatedly performing interval contractions and replacing ones with zeros.
\end{defn}

Note that being an interval minor is transitive in the sense that if $M_1$ is an interval minor of $M_2$, and $M_2$ is an interval minor of $M_3$, then $M_1$ is an interval minor of $M_3$.

We remark that one can also interpret a sequence of interval contractions in the following alternative way. For a zero-one matrix $M$, fix an interval partition of its set of rows and an interval partition of its set of columns, thereby interpreting $M$ as a block matrix. We can then define a contracted zero-one matrix with an entry for each block, where an entry is a zero if and only if its corresponding block is an all-zero matrix. 

Write $P_\pi$ for the permutation matrix of $\pi$, and note that a permutation $\sigma$ contains a pattern $\pi$ if and only if $P_\pi$ is an interval minor of $P_\sigma$. We say that a zero-one matrix $M$ \emph{contains} a permutation $\pi\in\mathcal{S}_k$ if $P_\pi$ is an interval minor of $M$. 

\begin{lem}\label{lem:matrix-dominated}
Let $\sigma$ be a uniform random permutation in $\mathcal{S}_n$, and let $M$ be a uniform random $(2k)\times m$ zero-one matrix, where\footnote{To be fully rigorous we should assume that $n$ is divisible
by $4k$.
Such divisibility considerations will
be inconsequential throughout the paper, and we do not discuss them further.} $m \coloneqq n/(4k)$. Then we can couple $\sigma$ and $M$ in such a way that $M$ is always an interval minor of $P_\sigma$.
\end{lem}
\begin{proof}
First, we observe that a uniform random permutation $\sigma\in\mathcal{S}_{n}$
can be obtained via a sequence of $n$ i.i.d.\ $\Unif\left(0,1\right)$
random variables $U_{1},\dots,U_{n}$ (whose values are distinct with probability
1), by taking $\sigma$ to be the unique permutation for which $U_{\sigma\left(1\right)}<\dots<U_{\sigma\left(n\right)}$.

We divide the interval $\left[0,1\right]$ into $2k$ consecutive
equal-sized intervals $I_{1},\dots,I_{2k}$ (so $I_y$ is the interval  between $(y-1)/(2k)$ and ${y}/{(2k)}$), and we divide the discrete
interval $\left\{ 1,\dots,n\right\} $ into $m:=n/\left(4k\right)$
consecutive (discrete) equal-sized intervals $J_{1},\dots,J_{m}$ (so $J_x$ contains the integers from $4k(x-1)+1$ to $4kx$ inclusive). Let $M_U$ be the random $\left(2k\right)\times m$
matrix with $\left(y,x\right)$-entry\footnote{Here $x$ represents the column index (i.e.,\ the horizontal coordinate), and $y$ represents the row index (i.e.,\ the vertical coordinate). It is rather unfortunate that the accepted convention for indexing matrices is opposite to the convention for indexing points in 2-dimensional space.}
\[
M_U\left(y,x\right)=\begin{cases}
1 & \text{if there is }j\in J_{x}\text{ with }U_{j}\in I_{y},\\
0 & \text{otherwise}.
\end{cases}
\]
Observe that $M_U$ is an interval minor of $P_\sigma$. It remains to check that there is a coupling between $M_U$ and $M$ such that $M\le M_U$. Since the columns of $M_U$ are i.i.d., we just need to show that the first column of $M$ is stochastically dominated by the first
column of $M_U$. Consider any $y\in\left[2k\right]$, and condition
on any outcome of the values of $M_U\left(y',1\right)$ for $\left(y',1\right)\ne\left(y,1\right)$.
It suffices to show that, conditionally, we have $M_U\left(y,1\right)=1$
with probability at least $1/2$.

To see this, note that to reveal whether $M_U\left(y',1\right)=1$,
it suffices to run through the indices $i\in J_{1}$, and keep checking
whether $U_i \in I_{y'}$ until we first see a success. So, after
revealing this information for all $y'\ne y$, there are at least
$4k-\left(2k-1\right)\ge2k$ indices $i\in J_{1}$ for which $U_i$ could still lie in $I_y$. In fact, each such $U_i$ is conditionally uniform on some set containing $I_y$, meaning that $U_i\in I_y$ with probability at least $1/(2k)$. So, the
conditional probability of the event $M_U\left(y,1\right)=0$ is at
most $\left(1-1/\left(2k\right)\right)^{2k}\le1/e\le1/2$, as desired.
\end{proof}

We say that a matrix $M$ is \emph{$k$-universal} (with respect to permutations from $\S_k$) if it contains every $\pi \in \mathcal{S}_k$. The above lemma shows that in order to prove that a uniform random $\sigma \in \mathcal{S}_n$ is $k$-universal, it suffices to show that a random $(2k)\times m$ matrix $M$ is $k$-universal, where $m=n/(4k)$.

\begin{rem*}
The matrix $M$ is typically ``dense'' (about half of its entries are ones), and one may wonder whether this density alone is enough to ensure that $M$ is $k$-universal
(provided $m\ge2k$, say). Although this suffices for the containment
of certain permutations such as the identity $1_k\in\mathcal{S}_{k}$, Fox~\cite[Theorem~6]{Fox13}
established the existence of a matrix, almost all of whose entries are ones, which fails to contain
almost all $\pi\in\mathcal{S}_{k}$. Thus, the density of ones alone is not enough to guarantee the $k$-universality of $M$.
\end{rem*}

Fix $\pi \in \mathcal{S}_k$. We now describe a procedure for finding
a copy of $\pi$ in a random $M$. We will use this procedure in the proofs
of \cref{thm:main,thm:almost-all}.

For each $t\in [k]$, we attempt to find a copy of $\pi$ in rows\footnote{It would be more natural to also consider $t=0$, since otherwise we actually never touch the first row of the matrix. However considering only $t\in [k]$ makes the indexing slightly more convenient.} $t+1,\dots,t+k$ in the following greedy fashion. First scan through row $\pi\left(1\right)+t$
from left to right until a one is found in some position $(\pi\left(1\right)+t,x_1)$, then scan through row $\pi\left(2\right)+t$,
starting from column $x_{1}+1$, until a one is found in some position $(\pi\left(2\right)+t,x_2)$, and so on (see \cref{fig:thread-scanning} below). We call this
procedure ``scanning along thread $t$ to find $\pi$''. Note that thread $t$ successfully finds $\pi$ if and only if some copy of $\pi$ lies in rows $t+1,\dots,t+k$, and it exposes at most $m$ entries of $M$ since it checks at most one entry in each column.

\begin{figure}[h]
\centering
\[
\begin{pmatrix}
* & * & * & * & * & * & * & * & * & * & * & * \\
* & * & * & * & * & * & * & * & * & * & * & * \\
0 & 1 & * & * & * & * & * & * & * & * & * & * \\
* & * & * & * & * & * & * & * & 0 & 0 & 1 & * \\
* & * & 0 & 0 & 0 & 0 & 0 & 1 & * & * & * & * \\
* & * & * & * & * & * & * & * & * & * & * & *
\end{pmatrix}
\]
\caption{One possible outcome for thread $2$ successfully finding a copy of $\pi = 132$ in rows $\{3,4,5\}$ of a $6 \times 12$ random matrix $M$. Starred entries remain unexposed.}
\label{fig:thread-scanning}
\end{figure}

\begin{rem*}
Since submitting the first version of this paper, it was brought to our attention that the same multi-threaded scanning procedure was recently also considered by Cibulka and Kyn\v{c}l~\cite{cibulka}, in a different context.
\end{rem*}

We conclude this section with a simple lemma that will be useful for analyzing our scanning process. Given coordinates $\left(y,x\right)\in\left[2k\right]\times\left[m\right]$,
let 
\[
r\left(y,x\right)=\max\left\{ t\in\left[m\right]:M\left(y,x+s\right)=0\text{ for }0\le s\le t\right\} 
\]
be the length of the run of zeros in $M$ starting at $\left(y,x\right)$
and continuing left-to-right (so $r\left(y,x\right)=0$ if $M\left(y,x\right)=1$).
Note that $r\left(y,x\right)$ has a geometric distribution supported
on the nonnegative integers, with success probability $1/2$. The following lemma shows that long runs of zeros
are unlikely, in a fairly general sense.

\begin{lem}
\label{lem:run-tail}There is an absolute constant $c>0$ such that
the following holds for sufficiently large $k$. If $\left(y_{1},x_{1}\right),\dots,\left(y_{\ell},x_{\ell}\right)$
are $\ell\le k$ positions in $M$, in distinct rows, then for $r\ge4\ell$
we have 
\[
\Pr\left(r\left(y_{1},x_{1}\right)+\dots+r\left(y_{\ell},x_{\ell}\right)\ge r\right) < e^{-r/8}.
\]
\end{lem}

\begin{proof}
As the $\ell$ positions $(y_i, x_i)$ all lie in distinct rows, the run lengths $r(y_1, x_1),\ldots, r(y_\ell, x_\ell)$ are independent random variables. As observed above, they are individually geometric random variables. Thus the the sum $r\left(y_{1},x_{1}\right)+\dots+r\left(y_{\ell},x_{\ell}\right)$ has a negative binomial distribution, and the desired inequality follows directly from a concentration inequality for the negative
binomial distribution. See \cite[Problem~2.5]{DP09}.
\end{proof}

\section{\label{sec:quasirandom}Quasirandom permutations}

In this section we prove \cref{thm:almost-all}. The ideas will also
be relevant for \cref{thm:main}. 

For $\pi\in\mathcal{S}_{k}$, we
say that a subset $A\subseteq[k]$ is a \emph{$\Delta$-shift} of another subset $B\subseteq[k]$
in $\pi$ if, writing $a_{1}<\dots<a_{L}$ for the elements of $A$
and $b_{1}<\dots<b_{L}$ for the elements of $B$, we have $\pi\left(a_{i}\right)=\pi\left(b_{i}\right)+\Delta$
for each $i$. For $\pi\in\mathcal{S}_{k}$, let $L_{\Delta}\left(\pi\right)$
be the largest $L$ for which there are $L$-sets $A,B\subseteq[k]$
such that $A$ is a $\Delta$-shift of $B$. Equivalently, $L_{\Delta}(\pi)$
is the length of the longest increasing subsequence of the function
$i\mapsto\pi^{-1}(\pi(i)+\Delta)$, where $i$ ranges over all indices
for which $\pi(i)\le k-\Delta$.

\begin{figure}[h]
\centering
\[
\begin{pmatrix}
* & * & * & * & * & * & * & * & * & * & * & * \\
* & * & * & * & * & * & * & * & * & * & * & * \\
0 & 1 & * & * & * & * & * & * & * & * & * & * \\
{\bf 0} & {\bf 0} & {\bf 0} & {\bf 0} & {\bf 1} & * & * & * & * & 0 & 0 & 0 \\
* & * & 0 & 0 & 0 & 0 & {\bf 0} & {\bf 0} & {\bf 1} & * & * & * \\
* & * & * & * & * & {\bf 1} & * & * & * & * & * & *

\end{pmatrix}
\]
\caption{One possible outcome for threads $2$ and $3$ searching for a copy of $\pi = 132$ in a $6 \times 12$ random matrix $M$, where thread $2$ fails but thread $3$ succeeds. Bolded entries are those checked by thread $3$. Note that because $L_1(132)=1$, the two threads can intersect in at most one row (in this case row $5$).}
\label{fig:L-Delta}
\end{figure}

The purpose of this definition is that the values $L_{\Delta}\left(\pi\right)$
measure the extent to which threads $t$ and $t+\Delta$ can intersect
each other (see Figure~\ref{fig:L-Delta}), in the multi-threaded scanning procedure described in
\cref{sec:multi-threaded-scanning}. The following lemma shows that
for almost all $\pi\in\mathcal{S}_{k}$, each of the values $L_{\Delta}\left(\pi\right)$
is quite small.
\begin{lem}
\label{lem:random-L}W.h.p, a uniform random permutation $\pi\in\mathcal{S}_{k}$
satisfies $L_{\Delta}\left(\pi\right)\le3\sqrt{k}$ for all $\Delta\in[k]$.
\end{lem}

To prove \cref{lem:random-L}, it will be convenient to make a few
definitions. For a function $f$ and a subset $A$ of its domain,
we use the notation $f|_{A}$ for the restriction of $f$ to $A$. Also, for subsets $A,B\subseteq\left[k\right]$ with elements $a_{1}<\dots<a_{L}$
and $b_{1}<\dots<b_{L}$, let $G\left(A,B\right)$ be the graph on
the vertex set $A\cup B$ with edge set $\{a_ib_i\}_{i=1}^L$. In particular, if $A\cap B = \emptyset$ then $G(A,B)$ is a matching, and in general $G(A,B)$ is always a vertex-disjoint union of paths. Whenever we use this definition we we will have $a_i\ne b_i$ for all $i$, so that $G(A,B)$ has no loops.
\begin{proof}[Proof of \cref{lem:random-L}]
For a fixed subset $B\subseteq[k]$ of order $L$, we bound the
probability that there exists a $\Delta$-shift of $B$ in $\pi$.
Such a shift may only exist if $\pi(B)\subseteq [k-\Delta]$, so we can assume
this. Now, if a shift of $B$ exists, it must be the set $A:=\pi^{-1}\left(\pi\left(B\right)+\Delta\right)$,
consisting of all indices $a$ such that $\pi\left(a\right)=\pi\left(b\right)+\Delta$
for some $b\in B$. Condition on any outcome of the random set $A$,
and on an outcome of $\pi|_{B\setminus A}:B\setminus A\to\left[k\right]$.

Now, conditionally, $\pi|_{A}$ is uniformly random among the $L!$ bijections from
$A$ into $\pi\left(B\right)+\Delta$. But, since we have conditioned
on an outcome of the function $\pi|_{B\setminus A}$, observe that there is only
one possibility for $\pi|_{A}$ that results in $A$ being a $\Delta$-shift
of $B$ in $\pi$. Indeed, note that the graph $G\left(A,B\right)$
is a disjoint union of paths, and that if $A$ is a $\Delta$-shift
of $B$ then $\pi|_{A\cup B}$ is fully determined by specifying the
value of $\pi\left(b\right)$ for a representative $b$ from each
component of $G\left(A,B\right)$. Since each path of $G\left(A,B\right)$
has an endpoint in $B\setminus A$, specifying $\pi|_{B\setminus A}$
determines $\pi|_{B}$.

It follows that
\[
\Pr\left(\max_{\Delta}L_{\Delta}\left(\pi\right)\ge L\right)\le k\binom{k}{L}\frac{1}{L!}\le k\Big(\frac{e^{2}k}{L^{2}}\Big)^{L},
\]
since there are $k$ choices of $\Delta$ and $\binom{k}{L}$ choices
for $B$. This probability is $o\left(1\right)$ for $L=3\sqrt{k}$,
which completes the proof.
\end{proof}
Now, let $\mathcal{\mathcal{Q}}_{k}\subseteq\mathcal{S}_{k}$ be the
set of $\pi\in\mathcal{S}_{k}$ such that $L_{\Delta}\left(\pi\right)\le3\sqrt{k}$
for each $\Delta\in\left[k\right]$. By \cref{lem:random-L}, we have
$|\mathcal{Q}_{k}|=(1-o(1))k!$. We are ready to prove \cref{thm:almost-all}.
\begin{proof}[Proof of \cref{thm:almost-all}]
Let $m=5k$, and consider a uniform random $\left(2k\right)\times m$
zero-one matrix $M$, whose entries are independently zero or one,
each with probability $1/2$. We will show that w.h.p.\ $M$ contains
every pattern $\pi\in\mathcal{Q}_{k}$. By \cref{lem:matrix-dominated}, the desired result follows:
w.h.p.\ a uniform random permutation of length $n=20k^{2}$ contains
every pattern in $\mathcal{Q}_{k}$.

Now, recall that $r\left(y,x\right)$ is the length of the longest
run of zeros in $M$ starting at $\left(y,x\right)$. Applying
\cref{lem:run-tail} with $\ell=1$, for each $\left(y,x\right)$ we
have $\Pr\left(r\left(y,x\right)>\log^{2}k\right)=o\left(1/(km)\right)$. Let $\mathcal A$ be the event that $r\left(y,x\right)\le\log^{2}k$ for all $(y,x)\in [2k]\times[m]$, so $\mathcal A$ holds w.h.p., by the union bound.

We wish to run $\log^{2}k$ threads of the scanning procedure
described in \cref{sec:multi-threaded-scanning}, for each $\pi\in\mathcal{Q}_{k}$.
However, we make a small modification to the procedure: if, during
some thread, we scan along a row for $\log^{2}k$ steps, finding only
zeros, then we pretend that the last of the entries checked was actually
a one (and continue scanning through some other row to find more of $\pi$).
We say that the thread \emph{succeeds} if it thinks it found a copy
of $\pi$ under these pretensions, and otherwise we say it \emph{fails}.
The plan is to show that for each $\pi\in\mathcal{Q}_{k}$, the probability
that all of our $\log^{2}k$ threads fail is only $e^{-\Omega(k\log^{2}k)}=o\left(1/\left|\mathcal{Q}_{k}\right|\right)$,
and then observe that 
\begin{align*}
\Pr\left(M\text{ contains every pattern in }\mathcal{Q}_{k}\right) & \ge\Pr\left(\text{\ensuremath{\mathcal{A}}}\cap\left\{ \text{for each }\pi\in\mathcal{Q}_{k}\text{, some thread succeeds}\right\} \right)\\
 & \ge\Pr\left(\mathcal{A}\right)-\sum_{\pi\in\mathcal{Q}_{k}}\Pr\left(\text{each thread fails for }\pi\right)=1-o\left(1\right).
\end{align*}
So, it suffices to fix $\pi\in\mathcal{Q}_{k}$ and show that with
probability $1-o\left(1/\left|\mathcal{Q}_{k}\right|\right)$ at least
one of the $\log^{2}k$ threads succeeds. Let $T_{t}\subseteq\left[2k\right]\times\left[5k\right]$
be the set of entries exposed by thread $t$. For $\pi\in\mathcal{Q}_{k}$
and $t\le\log^{2}k$, let $\mathcal{E}_{t}$ be the event that thread $t$ fails (in which case $|T_{t}|=5k$).

Now, observe that the intersections of threads always satisfy $|T_{t}\cap T_{t'}|\le3\sqrt{k}\log^{2}k$.
Indeed, if $t<t'$ and $X\subseteq\left[2k\right]$ is the set of
rows on which $T_{t}$ and $T_{t'}$ intersect, then $\pi^{-1}(X-t)$
is a $(t'-t)$-shift of $\pi^{-1}(X-t')$ in $\pi$, so $T_{t}\cap T_{t'}$
can intersect in at most $L_{t'-t}(\pi)\le3\sqrt{k}$ rows, each of
which can contain at most $\log^{2}k$ entries of $T_{t}\cap T_{t'}$.

It follows that, if $\mathcal{E}_{t}$ occurs, then 
\[
\left|T_{t}\setminus\left(T_{1}\cup\dots\cup T_{t-1}\right)\right|\ge5k-(t-1)3\sqrt{k}\log^{2}k\ge4k
\]
for sufficiently large $k$. That is to say, if thread $t$
fails, then it runs through at least $4k$ entries that were not exposed
by previous threads. Also, if a thread ever finds $k$ ones then it succeeds. So, the
probability that thread $t$ fails, conditioned on any outcome
of the previous threads, is upper-bounded by the probability that
a sequence of $4k$ coin flips results in fewer than $k$ heads, which
is at most $e^{-k/2}$ by a Chernoff bound (see for example \cite[Theorem~1.1]{DP09}).
That is to say,
\[
\Pr\left(\mathcal{E}_{t}\cond\mathcal{E}_{1}\cap\dots\cap\mathcal{E}_{t-1}\right)\le e^{-k/2},
\]
which implies that
\[
\Pr\left(\mathcal{E}_{1}\cap\dots\cap\mathcal{E}_{\log^{2}k}\right)\le e^{-\left(k/2\right)\log^{2}k}=o\left(1/\left|\mathcal{Q}_{k}\right|\right),
\]
as desired.
\end{proof}

\section{Structure vs quasirandomness}\label{sec:main}

In this section we prove \cref{thm:main}. We first define a notion of quasirandomness, in a similar spirit to the definition of $\mathcal Q_k$ in the previous section.
\begin{defn}
\label{def:quasirandomness}For $\pi\in\mathcal{S}_{k}$ and a subset
$X\subseteq\left[k\right]$, let $L_{\Delta}\left(\pi,X\right)$ be
the largest $L$ such that there are $L$-sets $A\subseteq X$ and $B\subseteq\left[k\right]$ with $A$ being a $\Delta$-shift of $B$. Say that
$X$ is \emph{$\left(\alpha,q\right)$-quasirandom in $\pi$} if $L_{\Delta}\left(\pi,X\right)\ge\alpha k$
for fewer than $q$ values of $\Delta\in[k]$.
\end{defn}

This quasirandomness condition is designed to ensure that if we choose
a sequence of threads $t_{1},\dots,t_{\ell}$ in such a way that the
pairwise differences $t_{j}-t_{i}$, for $i<j$, are not among the
small number of ``exceptional'' values of $\Delta$, then the quasirandom part of any thread cannot intersect very much with the other threads.

Next, for a set $X\subseteq\left[k\right]$, let $\mathcal{S}_{X,k}$
be the set of injections $\pi:X\to\left[k\right]$, and write $\S_{k}^{*}=\bigcup_{X\subseteq\left[k\right]}\S_{X,k}$.
Our main lemma shows that there is a relatively small family $\mathcal{Z}_{k}\subseteq\S_{k}^{*}$
of ``structured'' maps, having the property that every permutation
$\pi$ can be decomposed into a quasirandom part and a structured
part. Here and in the rest of the section, we assume that $k$ is sufficiently large.
\begin{lem}
\label{lem:structure-vs-randomness}There exists a family $\mathcal{Z}_{k}\subseteq\S_{k}^{*}$
such that:
\end{lem}

\begin{enumerate}
\item [(1)]$|\mathcal{Z}_{k}|\le e^{21k\log\log k}$, and
\item [(2)]for every $\pi\in\S_{k}$, there is a partition $Q\cup Z=\left[k\right]$
such that $Q$ is $(\log^{-4}k,\log^{5}k)$-quasirandom in $\pi$,
and $\pi|_{Z}\in\mathcal{Z}_{k}$.
\end{enumerate}

In the proof of \cref{lem:structure-vs-randomness} we will give a
precise definition of $\mathcal{Z}_{k}$, but the details of this
definition are not important for the proof of \cref{thm:main}. Thus, before
proving \cref{lem:structure-vs-randomness} we deduce \cref{thm:main}
from it. As outlined in \cref{sec:outline}, the plan is to first use \cref{lem:run-tail} to show that our random matrix $M$ is likely to have a property that makes multi-threaded scanning especially efficient on the structured part of a permutation $\pi$. We then proceed in a similar way to the proof of \cref{thm:almost-all}.

\begin{proof}[Proof of \cref{thm:main}]
For readability, let $C=21$, so $|\mathcal{Z}_{k}|\le e^{Ck\log\log k}$.
Let $M$ be a uniform random $(2k)\times m$ zero-one matrix,
for $m=17Ck\log\log k$. By \cref{lem:matrix-dominated} it suffices
to show that w.h.p.\ $M$ is $k$-universal (note that $4km\le 2000k^2\log \log k$). We now define two events that will be helpful in controlling the behavior of the multi-threaded scanning procedure described in \cref{sec:multi-threaded-scanning}.

Let $\mathcal{A}$ be the
event that $r\left(y,x\right)<\log^{2}k$ for every $\left(y,x\right)\in[m]\times[2k]$,
and observe that by \cref{lem:run-tail}, $\Pr(r(y,x) \ge \log^2 k) < e^{-\Omega(\log^2 k)}$. Thus, $\Pr\left(\mathcal{A}\right)=1-o\left(1\right)$. Recall that if $\mathcal A$ holds, then in our multi-threaded scanning procedure, no thread intersects any row in too many entries.

We also define a similar event which controls the amount of time a thread spends on structured maps $\phi\in \mathcal Z_k$. Let $\mathcal{B}$
be the event that for every $\phi \in \mathcal{Z}_k$, every $1\le x_{1}<\dots<x_{\ell}\le m$, every
$1\le y_{1}<\dots<y_{\ell}\le2k$ and every $t\in\left[k\right]$
we have $\sum_{i=1}^{\ell}r\left(x_{i},y_{\phi\left(i\right)}+t\right) < 16Ck\log\log k$.
By \cref{lem:run-tail}, for any individual choice of $\phi\in\mathcal{Z}_{k}$,
$t\in\left[k\right]$ and $x_{1},y_{1},\dots,x_{\ell},y_{\ell}$ we
have
\[
\Pr\left(\sum_{i=1}^{\ell}r\left(y_{i},x_{i}\right)\ge16Ck\log\log k\right) < e^{-2Ck\log\log k},
\]
so by the union bound over all choices of $\phi$, $t$, and $x_{1},y_{1},\dots,x_{\ell},y_{\ell}$, we have
\[
\Pr\left(\mathcal{B}\right)\ge1-\left|\mathcal{Z}_{k}\right|\cdot k \cdot \binom{2k}{k}\cdot\binom{m}{k}\cdot e^{-2Ck\log\log k}=1-o\left(1\right).
\]
Here we used the estimate $\binom{m}{k} \le (em/k)^k = e^{O(k\log\log\log k)}$. 

Our proof now follows the multi-threaded scanning procedure described
in \cref{sec:multi-threaded-scanning}, using a total of $\log^2 k$
threads as before. We make a similar modification as we did in \cref{sec:quasirandom}, where we pretend that $\mathcal{A}\cap\mathcal{B}$ holds: if, when scanning along a row, knowing that $\mathcal{A}\cap\mathcal{B}$
holds would allow us to deduce that the current entry is a one, then we pretend
that this next entry is in fact a one and move on to a different row
to find the next element of $\pi$. As long as $\mathcal{A}\cap\mathcal{B}$ holds,
this agrees with reality.

Fix a particular $\pi\in\S_{k}$ and let $Q\cup Z=\left[k\right]$
be the decomposition of $\pi$ given by \cref{lem:structure-vs-randomness}.
Let $F$ be the set of $\Delta\in[k]$ for which $L_{\Delta}(\pi,Q)\ge k\log^{-4}k$,
so by quasirandomness $|F|<\log^{5}k<k\log^{-2}k$. It is therefore possible to choose threads $t_{1}<\dots<t_{\log^{2}k}$ so that
no difference $t_{i}-t_{j}$ lies in $F$.

For $i\le\log^{2}k$, let $\mathcal{E}_{i}$ be the event that thread
$t_{i}$ fails to find a copy of $\pi$, under our pretensions. It suffices
to show that $\Pr(\mathcal{E}_{1}\cap\dots\cap\mathcal{E}_{\log^2 k})=o\left(1/k!\right)$.

Let $T_{i}$ be the set of entries
checked by thread $t_{i}$, and observe that if $\mathcal{E}_{i}$
occurs then $\left|T_{i}\right|=m$. Divide $T_{i}$ into subsets
$T_{i}^{Q}$ and $T_{i}^{Z}$ corresponding to the entries which are
checked to find $\pi|_{Q}$ and $\pi|_{Z}$ respectively. Now, crucially,
we always have $|T_{i}^{Z}|\le16Ck\log\log k$, because we are pretending that $\mathcal{B}$ holds: for any set of positions to place
the structured map $\pi|_Z$, the total length of the runs starting at those positions
is at most $16Ck\log\log k$.

Also, for any $1\le j < i \le\log^{2}k$, we have $L_{t_{i}-t_{j}}(\pi,Q)<k\log^{-4}k$
by quasirandomness and the choice of the $t_{i}$, so $T_{i}^{Q}$
and $T_{j}$ must intersect in fewer than $k\log^{-4}k$ distinct rows.
Since we are pretending that $\mathcal{A}$ holds, $T_{i}^{Q}\subseteq T_{i}$
contains at most $\log^{2}k$ entries in any given row, so we always
have $|T_{i}^{Q}\cap T_{j}|\le k\log^{-2}k$. Hence, if $\mathcal{E}_{i}$ occurs then
\[
\left|T_{i}\setminus\left(T_{1}\cup\dots\cup T_{i-1}\right)\right|\ge17Ck\log\log k-16Ck\log\log k-(i-1)k\log^{-2}k\ge4k.
\]
As in the proof of \cref{thm:almost-all}, if thread $i$
fails, then it runs through at least $4k$ entries that were not exposed
by previous threads, and at most $k-1$ of these entries have a one
in them. So, conditioned
on any outcome of the entries revealed by the previous threads, the probability of $\mathcal{E}_i$ is at most $e^{-k/2}$ by
a Chernoff bound, and we deduce
\[
\Pr\left(\mathcal{E}_{1}\cap\dots\cap\mathcal{E}_{\log^{2}k}\right)\le e^{-\left(k/2\right)\log^{2}k}=o\left(1/k!\right),
\]
as desired.
\end{proof}

\subsection{Structured maps}

It remains to define our family of structured maps $\mathcal{Z}_{k}$
and prove \cref{lem:structure-vs-randomness}. Basically, if quasirandomness
fails to hold then there are many pairs of sets $A,B$ which are shifts
of each other, and where these sets intersect they heavily constrain
the structure of $\pi$. This motivates our definition of $\mathcal{Z}_{k}$.
\begin{defn}
For $X\subseteq\left[k\right]$ and $L$-sets $A,B\subseteq X$ with
elements $a_{1}<\dots<a_{L}$ and $b_{1}<\dots<b_{L}$, let $G\left(A,B\right)$
be the graph on the vertex set $X$ with edge set $\{a_ib_i\}_{i=1}^L$. For $q,b\ge1$, a \emph{$(q,b)$-shift-system
for $\phi\in\S_{X,k}$ }is a choice of $\Delta_{1},\dots,\Delta_{q}\in [k]$
and sets $A_{1},B_{1},\dots,A_{q},B_{q}\subseteq X$ satisfying the
following properties.
\begin{itemize}
\item Each $A_{i}$ is a $\Delta_{i}$-shift of $B_{i}$ in $\phi$, and
\item the graph $\bigcup_{i=1}^{q}G\left(A_{i},B_{i}\right)$ has at most
$b$ connected components.
\end{itemize}
Then, let $\mathcal{Z}_{X,k}$ be the set of maps $\phi\in\S_{X,k}$
which have a $\left(q, b\right)$-shift-system for some $q, b$ satisfying
\[
(b+q)\log k+k(\log q+1)\le10k\log\log k,
\]
and let $\mathcal{Z}_{k}=\bigcup_{X\subseteq\left[k\right]}\mathcal{Z}_{X,k}$.
\end{defn}

We need to prove the two parts of \cref{lem:structure-vs-randomness} with this choice of $\mathcal{Z}_k$: first, that $\mathcal{Z}_{k}$ is not too large, and second, that all permutations
can be decomposed into a structured part and a quasirandom part.
\begin{proof}[Proof of \cref{lem:structure-vs-randomness}(1)]
Let $\phi\in\mathcal{Z}_{X,k}$ be a structured map having a $(q,b)$-shift-system
\[
\left(\Delta_{1},\dots,\Delta_{q},A_{1},B_{1},\dots,A_{q},B_{q}\right),
\]
and let $G_{i}=G\left(A_{i},B_{i}\right)$. Write $G = \cup_{i=1}^q G_i$. We first claim that to
specify $\phi$, it suffices to specify $q,b,X$ and the following data:
\begin{itemize}
\item the differences $\Delta_{1},\ldots,\Delta_{q}$,
\item subsets $A_{i}'\subseteq A_{i},B_{i}'\subseteq B_{i}$ (corresponding
to subgraphs $G_{i}':=G(A_{i}',B_{i}')\subseteq G_{i}$) for which
$G':=\bigcup_{i=1}^{q}G_{i}'$ is a spanning forest of $G$, and
\item the value of $\phi\left(v\right)$, for a single representative vertex
$v$ in each connected component of $G$.
\end{itemize}
Given the above data, the value of $\phi(x)$ can be determined for
every $x\in X$. Indeed, consider the representative $v$ of the component
of $x$ in $G'$, so that there is a unique path from $v$ to $x$ in
$G'$. Suppose the edges along this path come from the graphs $G_{i_{1}}',\dots,G_{i_{\ell}}'$.
Then the value of $\phi(x)$ must be $\phi(x)\pm\Delta_{i_{1}}\pm\Delta_{i_{2}}\pm\cdots\pm\Delta_{i_{\ell}}$,
where the sign of $\Delta_{i_{j}}$ is determined by whether the $j$-th
edge is oriented from $A_{i_{j}}'$ to $B_{i_{j}}'$ in the path from
$v$ to $x$.

Thus, to bound $\left|\mathcal{Z}_{k}\right|$ it suffices to count
the total number of ways to specify the above three pieces of data.
There are at most $k$ choices of $q$, $k^{q}$ choices of
$\Delta_{1},\ldots,\Delta_{q}$, and $k^{b}$ choices of the
values of $\phi(v)$ for each of the component representatives $v$. Since a forest on $\left|X\right|$ vertices
has at most $\left|X\right|-1\le k$ edges, the total number of edges among all the $G_i'$ is at most $k$, and so the number of choices
of the sets $A_{i}',B_{i}'$
is at most
\[
\sum_{\substack{L_{1},\dots,L_{q}\\
L_{1}+\dots+L_{q}\le k
}
}\prod_{i}\binom{\left|X\right|}{L_{i}}^{2}\le\sum_{\substack{L_{1},\dots,L_{q}\\
L_{1}+\dots+L_{q}\le k
}
}\prod_{i}\Big(\frac{ek}{L_{i}}\Big)^{2L_{i}}.
\]
Taking logarithms and applying Jensen's inequality to the concave
function $z\mapsto2z\log(ek/z)$, we find that
\[
\sum_{\substack{L_{1},\dots,L_{q}\\
L_{1}+\dots+L_{q}\le k
}
}\prod_{i}\Big(\frac{ek}{L_{i}}\Big)^{2L_{i}}\le\sum_{\substack{L_{1},\dots,L_{q}\\
L_{1}+\dots+L_{q}\le k
}
}\exp\Big(2k\log(eq)\Big)=\binom{k+q}{q}\exp\Big(2k\log\left(eq\right)\Big).
\]
We deduce that the number of choices of $\phi\in\mathcal{Z}_{X,k}$ having a $(q,b)$-shift-system is at most
\[
k^{q}\cdot k^{b}\cdot\binom{k+q}{q}\exp\Big(2k\log\left(eq\right)\Big)\le \exp\Big(2\big((q+b)\log k+k(\log q+1)\big)\Big)\le e^{20k\log\log k}.
\]
There are at most $2^k\cdot k^2\le e^{k\log \log k}$ ways to choose $X,b,q$, so we conclude that $|\mathcal Z_k|\le e^{21k\log \log k}$, as desired.
\end{proof}
Next, for \cref{lem:structure-vs-randomness}(2), in which
we need to find a structure-vs-randomness decomposition of every $\pi\in\mathcal{S}_{k}$,
we will iterate the following lemma. Say that $\phi\in\mathcal{S}_{X,k}$
is itself $\left(\alpha,q\right)$-quasirandom if $L_{\Delta}\left(\phi\right)\ge\alpha k$
for fewer than $q$ values of $\Delta\in[k]$ (this is analogous but
slightly different from the notion in \cref{def:quasirandomness} concerning
quasirandomness of a set of indices in a permutation).
\begin{lem}
\label{lem:weak-structure-vs-randomness}If $\phi\in\mathcal{S}_{X,k}$
is not $\left(\alpha,q\right)$-quasirandom, then there is $Y\subseteq X$
with $\left|Y\right|\ge\alpha k/2$ such that $\phi|_{Y}$ has a
$(q,k/q)$-shift-system.
\end{lem}

\begin{proof}
If $\phi$ is not $\left(\alpha,q\right)$-quasirandom, then there
are $1\le\Delta_{1}<\dots<\Delta_{q}\le k$ with each $L_{\Delta_{i}}\left(\pi\right)\ge\alpha k$.
For each $i$, let $A_{i}$ and $B_{i}$ be $(\alpha k)$-sets
for which $A_{i}$ is a $\Delta_{i}$-shift of $B_{i}$, and consider
the graph $G=\bigcup_{i=1}^{q}G\left(A_{i},B_{i}\right)$ on the vertex
set $X$. Note that each $G\left(A_{i},B_{i}\right)$ has maximum
degree at most 2, so $G$ has maximum degree at most $2q$. On the
other hand, $G$ has $\alpha kq$ edges, so has average degree $2\alpha kq/|X|$. The sum of the degrees which are at least half this average is at least $\alpha kq$, so there is a set $U$ of at least $(\alpha kq)/(2q)=\alpha k/2$ vertices with degree at least $\alpha kq/|X|$.

Every connected component which intersects $U$ has size at least
$\alpha kq/|X|$, so letting $b=(\alpha k/2)/(\alpha kq/|X|)=|X|/(2q)\le k/q$, the largest $b$ components of $G$ comprise at least $\alpha k/2$ vertices. Let $Y$ be the set of vertices in these components, and observe that $\phi|_{Y}$ has a
$(q,k/q)$-shift-system.
\end{proof}
We now prove part (2) of \cref{lem:structure-vs-randomness}.

\begin{proof}[Proof of \cref{lem:structure-vs-randomness}(2)]
Fix $\pi\in\S_{k}$. Our objective is to show that there is a partition
$Q\cup Z=\left[k\right]$ such that $Q$ is $(\log^{-4}k,\log^{5}k)$-quasirandom
in $\pi$, and $\pi|_{Z}\in\mathcal{Z}_{k}$.

The plan is to apply \cref{lem:weak-structure-vs-randomness}
repeatedly, continuing to extract structured parts from $\pi$ until
this is no longer possible. To be specific, we will
obtain sequences of vertex sets $\left[k\right]=X_{0}\supseteq X_{1}\supseteq\dots\supseteq X_{\ell}$
and $Y_{1},\dots,Y_{\ell}$, such that $Q:=X_{\ell}$ and $Z:=Y_{1}\cup\dots\cup Y_{\ell}$
satisfy the desired properties. Let $\alpha=\left(1/2\right)\log^{-4}k$
and $q=\log^{5}k$; the sets $X_{i}$ and $Y_{i}$ are
defined recursively as follows. For each $i\ge0$:
\begin{enumerate}
\item [(1)] If $X_{i}$ is $(\log^{-4}k,q)$-quasirandom in $\pi$, then
we stop (taking $\ell=i$).
\item [(2)] If $\pi|_{X_{i}}$ is not itself $\left(\alpha,q\right)$-quasirandom,
then by \cref{lem:weak-structure-vs-randomness} there is $Y_{i+1}\subseteq X_{i}$
with $\left|Y_{i+1}\right|\ge\alpha k/2$ such that $\pi|_{Y_{i+1}}$
has a $\left(q,k/q\right)$-shift-system. Set $X_{i+1}=X_{i}\setminus Y_{i+1}$.
\item [(3)] If neither of the previous cases hold, then $X_{i}$ is not $(\log^{-4}k,q)$-quasirandom in $\pi$, so there exist $q$
values of $\Delta$ for which $L_{\Delta}(\pi,X_{i})\ge k\log^{-4}k=2\alpha k$. Since $\pi|_{X_{i}}$ is itself $\left(\alpha,q\right)$-quasirandom, for at least one of these values of $\Delta$ we have $L_{\Delta}(\pi|_{X_{i}})<\alpha k$.
This implies that there are sets $A\subseteq X_{i}$ and $B\subseteq\left[k\right]$
of size $2\alpha k$ such that $A$ is a $\Delta$-shift of $B$,
and $\left|B\cap X_{i}\right|<\alpha k$. Then let $Y_{i+1}$ be
the set of all $a\in A$ such that $\pi^{-1}(\pi(a)-\Delta)\notin X_{i}$ (informally speaking, this is the set of all $a\in A$ which are ``paired'' with some $b\in B\setminus X_i$). Observe that $\left|Y_{i+1}\right|>2\alpha k-\alpha k=\alpha k$, and set $X_{i+1}=X_{i}\setminus Y_{i+1}$.
\end{enumerate}
At the end of this recursive construction, $Q=X_{\ell}$ is $(\log^{-4}k,\log^{5}k)$-quasirandom
in $\pi$. Since each $\left|Y_{i}\right|\ge\alpha k/2$, the set
$Z=Y_{1}\cup\dots\cup Y_{\ell}$ has size at least $\ell\alpha k/2$ (so $\ell\le 2/\alpha=4\log^4 k$).
Also, we can see by induction that for $Z_{i}:=Y_{1}\cup\dots\cup Y_{i}$,
each $\pi|_{Z_{i}}$ has an $\left(q_{i},b_{i}\right)$-shift system,
for some $q_{i}\le iq$ and $b_{i}\le ik/q$. Indeed,
suppose that $\pi|_{Z_{i-1}}$ has a $\left(q_{i-1},b_{i-1}\right)$-shift-system.
If $Y_{i}$ was defined via case 2, then $\pi|_{Y_{i}}$ has a $\left(q,k/q\right)$-shift-system,
and we can simply combine the two shift-systems to give a $\left(q_{i},b_{i}\right)$-shift
system for $\pi|_{Z_{i}}$, with $q_{i}=q_{i-1}+q$ and $b_{i}=b_{i}+k/q$.
If $Y_{i}$ was defined via case 3, then we can take $q_{i}=q_{i-1}+1$
and $b_{i}=b_{i-1}$, obtaining a shift-system for $\pi|_{Z_{i}}$
by adding $\Delta_{q_{i}}=\Delta$ and the sets $B_{q_{i}}=Y_{i}$
and $A_{q_{i}}=\pi^{-1}(\pi(B_{q_{i}})+\Delta)$ to our shift-system for $\pi|_{Z_{i-1}}$.
Note that $G\left(A_{q_{i}},B_{q_{i}}\right)$ consists of edges between
$Y_{i}$ and $Z_{i-1}$, meaning that enlarging the shift-system does
not create any new connected components in the associated graph.

We have proved that $\pi|_{Z}$ has a $\left(q_{\ell},b_{\ell}\right)$-shift
system, with $q_{\ell}\le\ell q$ and $b_{\ell}\le\ell k/q$.
Recalling that $\ell\le 4\log^4 k$, and the definitions
$\alpha=\left(1/2\right)\log^{-4}k$ and $q=\log^{5}k$, we observe
that
\[
(b_{\ell}+q_{\ell})\log k+k(\log q_{\ell}+1)\le10k\log\log k,
\]
meaning that $\pi|_{Z}\in\mathcal{Z}_{k}$.
\end{proof}

\section{\label{sec:concluding}Concluding remarks}

In this paper we proved that for $n=2000k^{2}\log\log k$, w.h.p.\ a
random $\sigma\in\mathcal{S}_{n}$ is $k$-universal. While the constant
2000 can clearly be improved, it seems that new ideas are necessary
for an asymptotic improvement. In particular, the bound $|\mathcal{Z}_k| \le e^{O(k\log \log k)}$ in \cref{lem:structure-vs-randomness} is best-possible: indeed, consider the family $\mathcal{L}_k$ of all permutations of length $k$ which can be decomposed into $\log^{10} k$ increasing subsequences of length $k\log^{-10} k$. Then $\mathcal L_k\subseteq \mathcal Z_k$ but $|\mathcal{L}_k| = e^{\Theta(k\log \log k)}$.

Also, one may naively hope that with a better structure-vs-randomness lemma it may be possible to strengthen the notion of ``structuredness'' to monotonicity. However, this is not possible, because a permutation can be extremely non-quasirandom and have no long monotone subsequences. Indeed, if $k=\ell^2$ and $\pi$ is the ``tilted grid'' permutation $a\ell + b + 1 \mapsto b\ell + a + 1$, for $0\le a,b<\ell$, then the longest increasing subsequence of $\pi$ has length $O(\sqrt{k})$, but $\pi$ is not even $(1/4, k/4)$-quasirandom.

A different direction towards \cref{conj:noga} is to directly
study the containment probabilities $\Pr\left(\pi\in\sigma\right)$.
Indeed, if one could show that for $\Pr\left(\pi\notin\sigma\right)=o\left(1/k!\right)$
for $n=\left(1+\varepsilon\right)k^{2}/4$, random $\sigma\in \S_n$, and any $\pi\in \S_k$, then \cref{conj:noga}
would follow directly from the union bound. One may naively conjecture the very strong bound $\Pr\left(\pi\notin\sigma\right)=e^{-\Omega(n)}$ for all $\pi$ (indeed, this is true for the identity permutation $\pi=1_k\in \S_k$, if say $n=2k^2$), but, perhaps surprisingly, using a construction of Fox~\cite[Theorem~6]{Fox13}
it is possible to show that when $n=k^{2+o\left(1\right)}$, for almost
all $\pi\in\mathcal{S}_{k}$ we have $\Pr\left(\pi\notin\sigma\right)\ge\exp\left(-k^{3/2+o\left(1\right)}\right)$. We conjecture that this bound is essentially tight.

\begin{conjecture}\label{conj:containment-strong}
If $n=1000k^2$, $\pi$ is a permutation of length $k$, and $\sigma$ is a uniform random permutation of length $n$, then 
\[
\Pr(\pi\not\in\sigma) \le \exp(-k^{3/2 +o(1)}).
\]
\end{conjecture}

Note that \cref{conj:containment-strong} would immediately imply that a typical $\sigma \in \mathcal{S}_n$ is $k$-universal for $n=1000k^2$. As some evidence for \cref{conj:containment-strong}, we observe that \cref{thm:almost-all} can be strengthened as follows.

\begin{prop}
Let $\mathcal{Q}_{k}\subseteq\mathcal{S}_{k}$ be the set in \cref{thm:almost-all}. If $n=20k^{2}$, and $\sigma$ is a uniform random permutation of length $n$, then
\[
\Pr(\pi\not\in\sigma) \le \exp(-\Omega(k^{5/4})).
\]
\end{prop}

To prove this, one can adapt the proof of \cref{thm:almost-all} in the following way. Instead of the event $\mathcal A$ (that the run lengths in $M$ are at most $\log^2 k$), we define the weaker event $\mathcal{A}'$ that the run lengths in $M$ are at most $k^{1/4}$, with up to $k/2$ exceptions. This event holds with probability $1-e^{-\Omega(k^{5/4})}$. Then, we condition on $\mathcal A'$, and delete the set of at most $k/2$ rows from $M$ which have exceptionally long runs of zeros, yielding a matrix $M'$ with at least $3k/2$ rows, in which no row has a run of length longer than $k^{1/4}$. For each $\pi\in \mathcal Q_k$, we then run $\Omega(k^{1/4})$ threads of our multi-threaded scanning procedure and show that it fails to find $\pi$ with probability at most $e^{-\Omega(k^{5/4})}$. We remark that this last step is more delicate than the corresponding argument in \cref{sec:quasirandom}, since the scanning procedure depends on the event we are conditioning on in a more complicated way.

The above result shows that a bound halfway to \cref{conj:containment-strong} from the trivial bound $e^{-\Omega(k)}$ holds for ``very quasirandom'' permutations. It is also easy to check that $\Pr(1_k \in \sigma) \le \exp(-\Omega(k^2))$, so \cref{conj:containment-strong} holds for the ``most structured'' permutation $1_k$ as well. What seem most difficult are ``hybrid'' permutations like the tilted square and the members of the family $\mathcal{L}_k$ defined earlier in this section, which are neither ``very quasirandom'' nor ``very structured''.

There are many other facts about containment of a single permutation that appear to be unknown. For example, the following weakening of \cref{conj:noga} appears to be open.

\begin{conjecture}\label{conj:noga-weak}Fix $\varepsilon>0$, and let $n=\left(1+\varepsilon\right)k^{2}/4$. Consider $\pi\in\S_k$, and let $\sigma\in S_{n}$ be a random permutation of order $n$.
Then w.h.p.\ $\sigma$ contains $\pi$.
\end{conjecture}

More generally, for a given permutation $\pi\in\mathcal S_k$, it would be interesting to understand the ``threshold'' value of $n$ above which a random $\sigma\in \S_n$ typically contains $\pi$. It is easy to see that such a threshold must always lie somewhere between $(1/e^2-o(1))k^2$ and $(1+o(1))k^2$, and it seems plausible that the threshold is about $k^2/e^2$ for almost all $\pi\in\mathcal S_k$.

\vspace{3mm}
\textbf{Acknowledgments. }We would like to thank Asaf Ferber for suggesting the crucial idea of adaptively exposing information multiple times to ``amplify'' containment probabilities, for helpful suggestions on the writing of the paper, and for several other stimulating discussions. The first author would also like to thank Jacob Fox, Ben Gunby, and Huy Pham for helpful conversations.


\begin{thebibliography}{10}

\bibitem{AD95}
D.~Aldous and P.~Diaconis, \emph{Hammersley's interacting particle process and
  longest increasing subsequences}, Probab. Theory Related Fields \textbf{103}
  (1995), no.~2, 199--213.

\bibitem{Alo16}
N.~Alon, \emph{When are random permutations $k$-universal?, {C}ombinatorics and
  probability}, Oberwolfach Rep. \textbf{13} (2016), no.~2, 1189--1257,
  Abstracts from the workshop held April 17--23, 2016, Organized by B\'{e}la
  Bollob\'{a}s, Michael Krivelevich, Oliver Riordan and Emo Welzl.

\bibitem{Alo17}
N.~Alon, \emph{Asymptotically optimal induced universal graphs}, Geom. Funct.
  Anal. \textbf{27} (2017), no.~1, 1--32.

\bibitem{alstrup2015adjacency}
S.~Alstrup, H.~Kaplan, M.~Thorup, and U.~Zwick, \emph{Adjacency labeling
  schemes and induced-universal graphs}, Proceedings of the forty-seventh
  annual ACM symposium on Theory of computing, ACM, 2015, 625--634.

\bibitem{Arr99}
R.~Arratia, \emph{On the {S}tanley-{W}ilf conjecture for the number of
  permutations avoiding a given pattern}, Electron. J. Combin. \textbf{6}
  (1999), Note, N1, 4.

\bibitem{bona2016combinatorics}
M.~B{\'o}na, \textbf{Combinatorics of permutations}, Chapman and Hall/CRC, 2016.

\bibitem{cibulka} J. Cibulka and J. Kyn\v{c}l, \emph{Better upper bounds on the F\"{u}redi-Hajnal limits of permutations}, Proceedings of the Twenty-Eighth Annual ACM-SIAM Symposium on Discrete Algorithms, SIAM, 2017, 2280--2293.

\bibitem{claesson2008classification}
A.~Claesson and S.~Kitaev, \emph{Classification of bijections between 321-and
  132-avoiding permutations}, S{\'e}m. Lothar. Combin \textbf{60} (2008), B60d.

\bibitem{DP09}
D.~P. Dubhashi and A.~Panconesi, \textbf{Concentration of measure for the
  analysis of randomized algorithms}, Cambridge University Press, Cambridge,
  2009.

\bibitem{EV18}
M.~Engen and V.~Vatter, \emph{Containing all permutations}, arXiv preprint
  arXiv:1810.08252 (2018).

\bibitem{EELW07}
H.~Eriksson, K.~Eriksson, S.~Linusson, and J.~W\"{a}stlund, \emph{Dense packing
  of patterns in a permutation}, Ann. Comb. \textbf{11} (2007), no.~3-4,
  459--470.

\bibitem{Fox13}
J.~Fox, \emph{Stanley-{W}ilf limits are typically exponential}, arXiv preprint
  arXiv:1310.8378 (2013).

\bibitem{LS77}
B.~F. Logan and L.~A. Shepp, \emph{A variational problem for random {Y}oung
  tableaux}, Advances in Math. \textbf{26} (1977), no.~2, 206--222.

\bibitem{marcus2004excluded}
A.~Marcus and G.~Tardos, \emph{Excluded permutation matrices and the
  Stanley--Wilf conjecture}, Journal of Combinatorial Theory, Series A
  \textbf{107} (2004), no.~1, 153--160.

\bibitem{Mil09}
A.~Miller, \emph{Asymptotic bounds for permutations containing many different
  patterns}, J. Combin. Theory Ser. A \textbf{116} (2009), no.~1, 92--108.

\bibitem{moon1965minimal}
J.~Moon, \emph{On minimal $n$-universal graphs}, Glasgow Mathematical Journal
  \textbf{7} (1965), no.~1, 32--33.

\bibitem{Rad64}
R.~Rado, \emph{Universal graphs and universal functions}, Acta Arith.
  \textbf{9} (1964), 331--340.

\bibitem{Rom15}
D.~Romik, \textbf{The surprising mathematics of longest increasing subsequences},
  Institute of Mathematical Statistics Textbooks, vol.~4, Cambridge University
  Press, New York, 2015.

\bibitem{Tao07b}
T.~Tao, \emph{Structure and randomness in combinatorics}, 48th Annual IEEE
  Symposium on Foundations of Computer Science (FOCS'07), IEEE, 2007,
  pp.~3--15.

\bibitem{VK77}
A.~M. Ver\v{s}ik and S.~V. Kerov, \emph{Asymptotic behavior of the {P}lancherel
  measure of the symmetric group and the limit form of {Y}oung tableaux}, Dokl.
  Akad. Nauk SSSR \textbf{233} (1977), no.~6, 1024--1027.

\end{thebibliography}
\end{document}